\documentclass[12pt]{article}
\usepackage[T1]{fontenc}
\usepackage{a4wide}

\usepackage{graphicx}
\usepackage{hyperref}
\usepackage{amsthm}
\usepackage{amsmath}
\usepackage{amssymb}
\usepackage{amsfonts}
\usepackage[noadjust]{cite}
\usepackage{epsfig}
\usepackage{mathtools}
\usepackage{enumitem}
\usepackage{stackengine,graphicx}
\usepackage{xspace}
\usepackage{float}
\usepackage{nicematrix}

 \newcommand{\vc}[1]{\langle #1 \rangle}

\newcommand{\F}{\mathbb{F}}
\renewcommand{\S}{\mathcal{S}}

\renewcommand{\leq}{\leqslant}

\newcommand{\Span}{\mathrm{span}}
\newcommand{\rows}{\mathrm{rows}}
\newcommand{\columns}{\mathrm{columns}}

\usepackage{todonotes}
\setuptodonotes{inline}

\newcommand\blfootnote[1]{%
  \begingroup
  \renewcommand\thefootnote{}\footnotetext{#1}%
  \endgroup
}

\newcommand{\ERCagreement}{
\blfootnote{\noindent
{\begin{minipage}[t]{0.70\textwidth}
\vspace{-15pt}
\small MP was supported by the project {\sc{BOBR}} that have received funding from the European Research Council (ERC) under the European Union's Horizon 2020 research and innovation programme (grant agreement No948057). 
KP was supported by the project GA24-11098S of the Czech Science Foundation.
JG was supported by the Polish National Science Centre SONATA-18 grant number 2022/47/D/ST6/03421.
\vspace{10pt}
 \end{minipage} \hspace{5pt}
 \begin{minipage}{.2\textwidth} \vspace{5pt} \hspace{25pt}
 \includegraphics[width=1.0\textwidth]{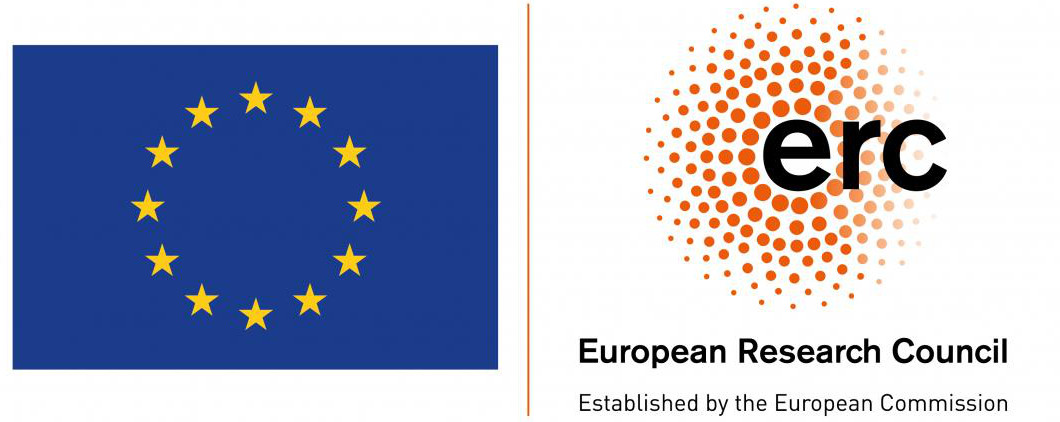}\end{minipage}\hfill}}}

\newcommand{\FF}{\mathbb{F}}

\newcommand{\QQ}{\mathbb{Q}}

\DeclareMathOperator{\td}{td}

\DeclareMathOperator{\cd}{cd}
\DeclareMathOperator{\dd}{dd}

\DeclareMathOperator{\csd}{c^{*}\hspace{-3pt}d}
\DeclareMathOperator{\dsd}{d^{*}\hspace{-3pt}d}

\DeclareMathOperator{\cbd}{c^{^{\bullet}}\hspace{-3.5pt}d}
\DeclareMathOperator{\dbd}{d^{^\bullet}\hspace{-3.5pt}d}

\newtheorem{theorem}{Theorem}
\newtheorem{corollary}[theorem]{Corollary}
\newtheorem{lemma}[theorem]{Lemma}
\newtheorem{proposition}[theorem]{Proposition}

\newtheorem{definition}[theorem]{Definition}

\newcommand{\FPT}{$\mathsf{FPT}$\xspace}

\newcommand{\Q}{\mathbb{Q}}

\def\ve#1{\mathchoice{\mbox{\boldmath$\displaystyle\bf#1$}}
{\mbox{\boldmath$\textstyle\bf#1$}}
{\mbox{\boldmath$\scriptstyle\bf#1$}}
{\mbox{\boldmath$\scriptscriptstyle\bf#1$}}}

\newcommand\veu{{\ve u}}
\newcommand\vev{{\ve v}}
\newcommand\vew{{\ve w}}

\def\N{\mathbb{N}}

\def \la {\langle}
\def \ra {\rangle}

\renewcommand{\phi}{\varphi}

\begin{document}

\title{Obstructions and dualities for matroid depth parameters}


\author{
Jakub Gajarský\thanks{Institute of Informatics, University of Warsaw, Poland. Email: \tt{gajarsky@mimuw.edu.pl}.}\and
Kristýna Pekárková\thanks{Faculty of Informatics, Masaryk University, Botanick\'a 68A, 602 00 Brno, Czech Republic. Email: \tt{kristyna.pekarkova@mail.muni.cz}.}\and
Michał Pilipczuk\thanks{Institute of Informatics, University of Warsaw, Poland. Email: \tt{michal.pilipczuk@mimuw.edu.pl}.}
}
\date{} 
\maketitle

\begin{abstract} 
    Contraction$^*$-depth is considered to be one of the analogues of graph tree-depth in the matroid setting.
    In this paper, we investigate structural properties of contraction$^*$-depth of matroids representable over finite fields and rationals. In particular, we prove that the obstructions for contraction$^*$-depth for these classes of matroids are bounded in size. From this we derive analogous results for related notions of contraction-depth and deletion-depth. Moreover, we define a dual notion to contraction$^*$-depth, named \emph{deletion$^*$-depth}, for $\FF$-representable matroids, and by duality extend our results from contraction$^*$-depth to this notion.
\end{abstract}

\ERCagreement

\section{Introduction}

A recent line of research in structural matroid theory
focuses on various matroid depth parameters and their applications. These parameters typically capture whether a matroid $M$ can be 
suitably encoded in a tree of bounded depth. Probably the most prominent examples are \emph{branch-depth}, introduced by DeVos, Kwon and Oum~\cite{DeVKO20}, and \emph{contraction$^*$-depth}\footnote{In the original paper~\cite{KarKLM17}, contraction$^*$-depth was called branch-depth. However, due to the naming clash with the notion by DeVos et al., this parameter was renamed to contraction$^*$-depth.}, introduced by Kardoš, Kráľ, Liebenau and Mach~\cite{KarKLM17}.

Recently it was shown that matroids of bounded contraction$^*$-depth can be used to find efficient algorithms for block-structured integer programs~\cite{ChaCKKP19,BriKKPS24}. This has led to further study of this notion and introduction of similar parameters such as \emph{deletion-depth}, \emph{contraction-depth}~\cite{DingO95,DeVKO20} and \emph{contraction$^*$-deletion-depth}~\cite{BriKL23}. All these parameters have the same general structure: The aim is to recursively decompose a matroid $M$ by applying a simplification operation (contraction, deletion, ...) to $M$ (if $M$ is connected) and then continuing in each connected component of the resulting matroid separately. The $X$-depth (where $X$ is the simplification operation in question) of $M$ corresponds to the number of applications of the simplification operation in the deepest branch in this recursive decomposition.

In this work, we focus on obstructions for these depth parameters. First, we  prove that  obstructions for contraction$^*$-depth of $\FF$-representable matroids are bounded in size:

\begin{theorem}
\label{thm:main1}
Let $\FF$ be a finite field.
Let $M$ be an $\mathbb{F}$-represented matroid such that
$\csd (M) = d$ and $\csd (M \setminus e) < d$ for every $e \in M$. Then $|M| \leq f(|\FF|, d)$ for some function $f$.
\end{theorem}

Our proof relies on techniques from finite model theory. We show that every $\FF$-represented matroid of bounded contraction$^*$-depth can be encoded in a tree $T$ of bounded depth using CMSO logic and also that the property of having contraction$^*$-depth exactly $d$ can be expressed in CMSO. Then we use a result of~\cite{GajH15} to argue
that if $T$ is tree that encodes matroid satisfying the assumptions of Theorem~\ref{thm:main1}, then $T$ cannot be too large.

Next, we introduce a new depth parameter named \emph{deletion$^*$-depth} and show that this parameter is dual to contraction$^*$-depth in the following sense.
\begin{theorem}
\label{thm:csd_dsd_dual}
    For every represented matroid $M$ we have $\csd(M) = \dsd(M^*)$.
\end{theorem}

We use this duality result to extend Theorem~\ref{thm:main_dd} to this new parameter.
\begin{theorem}
\label{thm:main_dsd}
Let $\FF$ be a finite field.
Let $M$ be an $\mathbb{F}$-represented matroid such that
$\dsd (M) = d$ and $\dsd (M/e) < d$ for every $e \in M$. Then $|M| \leq f(|\FF|, d)$ for some function $f$.
\end{theorem}

We then focus on contraction-depth and deletion-depth.
The proof Theorem~\ref{thm:main1} can be easily adjusted to obtain the following result for contraction-depth. 
\begin{theorem}
\label{thm:main_cd}
Let $\FF$ be a finite field.
Let $M$ be an $\mathbb{F}$-represented matroid such that
$\cd (M) = d$ and $\cd (M \setminus S) < d$ for every non-empty set $S$ of elements of $M$. Then $|M| \leq f(|\FF|, d)$ for some function $f$.
\end{theorem}

Using the fact that contraction and deletion are dual operations, we also obtain the following dual result for deletion-depth.
\begin{theorem}
\label{thm:main_dd}
Let $\FF$ be a finite field.
Let $M$ be an $\mathbb{F}$-represented matroid such that
$\dd (M) = d$ and $\dd (M/S) < d$ for every non-empty set $S$ of elements of $M$. Then $|M| \leq f(|\FF|, d)$ for some function $f$.
\end{theorem}

As a an example application of our Theorem~\ref{thm:main1} we obtain Theorem~\ref{thm:few_moves} below. For its statement, we note that if $M$ is a matroid represented by a matrix with $m$ rows, and if $\ve{v} \in \FF^m$, then by $M/\ve{v}$ we denote simplification operation used to define contraction$^*$-depth. 
We then call a vector $\ve{v} \in \FF^m$ \emph{progresive} if $\csd(M/\ve{v}) < \csd(M)$.

\begin{theorem}
\label{thm:few_moves}
    There exists a function $g: \N \times \N \to \N$ such that for every $\FF$ and $d$, if $M$ is a $\F$-represented matroid with $\csd(M) = d$, then the number of progressive vectors for $M$ is at most $g(|\FF|, d)$.
\end{theorem}

Finally, we note that Theorems~\ref{thm:main1} and~\ref{thm:few_moves} can be extended to $\QQ$-representable matroids with small entry complexity.

\paragraph{Related work.}
Characterizing obstructions is one of the classical problems surrounding
structural parameters. For matroids of branch-width at most $k$, it is known that 
each excluded minor has size at most $(6^{k+1} - 1)/5$ due to the result of 
Geelen, Gerards, Robertson and Whittle~\cite{GeeGRW03}. In~\cite{KanKKO23}, Kanté, Kim, Kwon and Oum studied obstructions for matroid 
path-width. Precisely, they proved that every excluded minor representable over a finite
field $\FF$ for the class of matroids of path-width at most $k$ has 
at most $2^{|\FF|^{O(k^2)}}$ elements. 

Various structural and algorithmic properties of contraction$^*$-depth were studied in~\cite{ChaCKKP20}, who studied the connection of  contraction$^*$-depth to integer programming, and gave an \FPT algorithm for computing contraction$^*$-depth of matroid representable over finite fields.
The closure properties of contraction-depth in relationship to contraction$^*$-depth was studied by Bria\'nski, Kr\'a\v{l} and Lamaison~\cite{BriKL23}, who proved that the minimum contraction-depth of a matroid $M'$
    that contains $M$ as a restriction is equal to the contraction$^*$-depth of $M$ increased by one.

A result analogous to our Theorem~\ref{thm:few_moves} was proven for graphs of bounded treedepth in~\cite{DvorakGT12}, and was used in~\cite{BoulandDK12} to show that the graph isomorphism problem is in FPT for graph classes of bounded treedepth.

\section{Preliminaries}

In this section, we recall the notation and preliminaries needed throughout the paper.

\subsection{Graph theory}

All graphs in this paper are finite, undirected and simple, without loops. We use standard graph theoretic notation and terminology. By a \emph{rooted tree} we mean a connected acyclic graph with a specified vertex, called the \emph{root} of the tree. All trees in this paper are rooted, unless stated otherwise. The \emph{depth} of a tree $T$ is the number of edges on the longest leaf-to-root path in $T$. Let $u,v$ be two vertices of a tree $T$. We say that $u$ is an ancestor of $v$ if $u$ lies on the unique path from $v$ to the root of $T$. The closure of a tree $T$, denoted by $cl(T)$, is the graph obtained by making every vertex of $T$ adjacent to all its ancestors (except itself).

A rooted forest is a graph in which each connected component is a rooted tree. The depth of a rooted forest is the largest depth of its connected components. The closure of a rooted forest $F$, denoted by $cl(F)$, is the graph obtained from $F$ by taking the closure of each connected component. 

Let $G$ be a graph. A rooted forest $F$ is a \emph{elimination forest} of $G$ if $V(G) = V(F)$ and $G$ is a subgraph of $cl(F)$. The \emph{tree-depth} of $G$, denoted by $\td(G)$, is defined as 
$$td(G):=\min_{F} depth(F) + 1,$$ where the minimum is taken over all elimination forests of $G$.


The following simple observation follows directly from the
definition of tree-depth.
\begin{lemma}
\label{lem:clique}
    Let $G$ be a graph and $F$ its elimination forest. Then in every clique $K$ of $G$, all vertices of $K$ are pairwise comparable by the ancestor relation of $F$. In particular, in every clique $K$ of $G$ there is a vertex $v$ sucht that all vertices of $K$ are  ancestors of $v$ in $F$.
\end{lemma}

\subsection{Linear algebra and matrices}
\label{sec:linalg}
We assume familiarity with vector spaces and basic notions of linear algebra. In this section, we mainly review the concepts and notation used in this paper. In what follows, $\F$ always denotes a field; usually this will be a finite field or $\Q$. All vector spaces we consider are finite-dimensional. Vectors are written in bold, while their coordinates are written in normal font. For example, $\vev$ denotes a vector, while $v_i$ its i-th coordinate.

Vectors $\ve{v_1},\ldots,\ve{v_k}$ from a vector space $V$ are \emph{linearly 
independent} if $\alpha_1 \ve{v_1} + \ldots + \alpha_k\ve{v_k} = \ve{0}$ implies $\alpha_1 = \ldots =\alpha_k = 0$; otherwise they are \emph{linearly dependent}.
Let $S$ be a set of vectors from a vector space $V$ over $\F$. By  $\Span(S)$ we denote the set of all $\F$-linear combinations of vectors from $S$, i.e. $\Span(S) := \{\alpha_1 \ve{v_1} + \ldots + \alpha_k \ve{v_k}~|~ k \in \N, \alpha_i \in \F, \ve{v_i} \in S\}$. Note that $\Span(S)$ is always a linear subspace of $V$.

For two vectors $\vev = (v_1,\ldots, v_n)^\top$ and $\vew = (w_1,\ldots, w_n)^\top$ in $\F^n$ we define $\vev\cdot \vew = \sum_{i=1}^n v_iw_i$.
Let $W$ be a linear subspace of $\F^n$. We define the \emph{orthogonal complement} of $W$ by $W^\bot := \{\vev \in \F^n~|~\vev \cdot \vew = 0 \text{ for all $\vew \in W$}\}$. Note that we have $(W^{\top})^\top = W$ and $(U \cap W)^\bot = \Span(U^\bot \cup W^\bot)$. Also, every $\ve{v} \in \F^n$ can be uniquely written as $ \ve{v'} +  \ve{v''}$ where $\ve{v'}\in W$ and $\ve{v''} \in W^\bot$. From this one easily shows that $W \cap \Span(\ve{v})^\bot = W \cap \Span(\ve{v'})^\bot$.


Let $V$ be a vector space over $\F$.
A set $B = \{\ve{v_1}, \ldots, \ve{v_k}\}$ of vectors is a \emph{basis} of $V$ if $B$ is inclusion-wise minimal such that $V = \Span(\ve{v_1}, \ldots, \ve{v_k})$. We say that a basis $B$ is \emph{orthogonal} if $\ve{v_i}\cdot \ve{v_j} = 0$ if and only if $i\not= j$ for each $i,j$. By the Gram-Schmidt orthogonalization, for any finite-dimensional vector space $V$ and any $\ve{v} \in V$ with $\ve{v}\not=\ve{0}$ we can find an orthogonal basis $\ve{v_1}, \ldots, \ve{v_k}$  of $V$  such that $\ve{v_1} = \ve{v}$.

Let $V$ be a vector space over $\F$ and $W$ be its linear subspace. 
The \emph{quotient space} $V/W$ has as its elements the equivalence classes of relation $\sim$ defined by $\ve{v} \sim \ve{u}$ if $\ve{v} - \ve{u} \in W$. The vector space structure $V/W$ is defined in the natural way. 




For a matrix $A$ over $\Q$, the \emph{entry complexity} of $A$, denoted by $ec(A)$, is the maximum length of a binary encoding of an entry. The length of a binary encoding of a rational number $r = p/q$ with $p$ and $q$ being coprime
is $\lceil \log_2(|p| + 1)\rceil + \lceil(\log_2
|q| + 1)\rceil$.

Finally, we define the \emph{tree-depth of matrices}. This is done by measuring the tree-depth of a graph associated with the matrix. There exist several ways to derive such underlying graph; in this paper, we focus on the concept of so called \emph{dual graph}. Formally, the dual graph of a matrix $A$, denoted by $G_D(A)$, is the graph where vertices one-to-one correspond to the rows of $A$ and there is an edge between two rows $r_1$ and $r_2$ if there exists a column $c$ such that both entries of $r_1$ and $r_2$ at $c$ are non-zero. The dual tree-depth of a matrix $A$ is then defined as $td_D(A):=td(G_D(A))$.


 
\subsection{Matroids}

For matroids, our presentation is primarily based on the monography of Oxley~\cite{Oxl11}.

A \emph{matroid} $M$ is a pair $(X, \mathcal{I})$, where $X$ is a finite set and $\mathcal{I} \subseteq 2^X$ a collection of subsets
that satisfies the following three \emph{matroid axioms}:
\begin{enumerate}
    \item $\emptyset \in \mathcal{I}$.
    \item If $A \in \mathcal{I}$ and $B \subseteq A$, then $B \in \mathcal{I}$.
    \item if $A, B \in \mathcal{I}$ and $|A| > |B|$, then there exists an element $x \in A \setminus B$ such that 
    $B \cup \{x\} \in \mathcal{I}$.
\end{enumerate}

The set $X$ is called the \emph{ground set} of $M$, and the sets in $\mathcal{I}$ are referred to as \emph{independent}. We usually refer to elements of the ground set of a matroid $M$ as \emph{elements} of $M$.

If $X' \subseteq X$ is a set of elements of a matroid $M$, then $M[X']$ denotes the \emph{restriction} of $M$ to $X'$. This is a matroid with the ground set begin $X'$, and a subset of elements of this matroid is independent if and only if it is independent in the original matroid $M$. For a subset $X' \subseteq X$, the \emph{rank} of $X'$, denoted by $r_M(X')$ is the largest size of an independent set contained in $X'$. The rank of $M$ is $r(M) = r_M(X)$.
A \emph{basis} of a matroid $M$ is a maximal independent set of the ground set, while 
 a \emph{circuit} of the matroid is a minimal dependent set.
A matroid $M$ is \emph{connected} if any two elements of $M$ are contained in a common circuit.
The property of being contained in a common circuit is transitive, that is, if two elements $e, e' \in M$ are contained in a common circuit and $e'$ and $e''$ are contained in a common circuit,
then also $e, e''$ are contained in a common circuit. A \emph{component} of a matroid $M$ is an inclusion-wise maximal connected restriction of $M$.

Let $A$ be a matrix over some field $\FF$ and $M(A)$ its vector matroid, i.e., a matroid, where the grounds set consists of the column vectors of $A$ and whose independent sets are the linearly independent sets of column vectors of $A$. If a matroid $M$ is isomorphic to $M(A)$, then we say that $M$ is \emph{representable} over $\FF$ or $\FF$\emph{-representable}. The matrix $A$ is then called the \emph{representation} for $M$ over $\FF$ or $\FF$\emph{-representation}. When a matroid $M$ is given by its representation $A$, we say that $M$ is \emph{$\F$-represented} or just \emph{represented}, if $\F$ is clear from the context or not important.


\subsubsection{Matroid contraction and deletion}

Let $M = (X, \mathcal{I})$ be a matroid and $X'\subseteq X$ a set of its elements. The \emph{restriction} of $M$ to $X'$, denoted by 
$M[X']$, is the matroid on the ground set $X'$ where a subset $X'' \subseteq X'$ is independent if it was independent in the original matroid $M$. By the \emph{deletion} of $X'$ from the matroid $M$ we mean the operation that takes $M$ to the matroid $M\setminus X':= M[X \setminus X']$.
To simplify the notation, if $X'$ is, in fact, just a single element, we write $M \setminus x$ rather than $M \setminus \{x\}$.

 The \emph{contraction} of a matroid $M$ by $X'$, denoted by $M/X'$, is the matroid with the ground set $X \setminus X'$ and the independence notion defined as follows: a subset $X'' \subseteq (X \setminus X')$ is independent in $M/X'$ if and only if:
 \begin{align*}
     r_M(X'' \cup X') = |X''| + r_M(X')
 \end{align*}

 For represented matroids, contraction can be defined as follows. Let $M = M(A)$ be a represented matroid and let $h$ be the number of rows of $A$. Further, let $K$ be a linear subspace of $\F^h$. We then define the matroid $M/K$ as the matroid whose ground set consists of the columns of $A$ and where columns $\ve{v_1},\ldots,\ve{v_k}$ are dependent if there exist $\alpha_1,\ldots, \alpha_k \in \F$, not all $0$, such that $\alpha_1\ve{v_1} + \ldots \alpha_k \ve{v_k} \in K$. If $K$ is a $1$-dimensional subspace of $\F^h$ spanned by a vector $\ve{v} \in \F^h$, we often write $M/\ve{v}$ instead of  $M/K$.  We remark that contraction for represented matroids is usually introduced using quotient spaces, but the more explicit (and equivalent) definition given above will be useful for us.




\subsubsection{Matroid duality}

Let $M = (X,\mathcal{I})$ be a matroid. We define the \emph{dual} of $M$, denoted by $M^*$, to be the matroid on the same ground set, and in which a set $S$ is independent if and only if $M$ has a basis $B$ disjoint from $S$.

We will need the following two facts about matroid duality.

\begin{lemma}[Deletion is dual to contraction]
  For every matroid $M$ and any $e$ we have $(M\setminus e)^* = M^*/e$.  
\end{lemma}

\begin{lemma}
\label{lem:dual_CCs}
    For a matroid $M$, a subset $S$ of its ground set is a connected component of $M$ if and only if it is a connected component of $M^*$.
\end{lemma}


\subsection{Matroid parameters and their properties}

\begin{definition}[\cite{DeVKO20}]
\label{def:dd}  
The \emph{deletion-depth} $\dd(M)$ of a matroid $M$ is defined recursively as:
\begin{itemize}
    \item If $M$ has single element, then $\dd(M) = 1$.
    \item If $M$ is not connected, then the deletion-depth of $M$ is the maximum deletion-depth of a component of $M$.
    \item If $M$ is connected, then $\dd(M) = 1 + \min_{e \in M} \dd(M\setminus e)$.
\end{itemize}
\end{definition}

For a matrix $A$, the deletion depth of $A$ is defined as $\dd(A) := \dd(M(A))$.

\begin{definition}[\cite{DeVKO20}]
\label{def:cd}  
The \emph{contraction-depth} $\dd(M)$ of a matroid $M$ is defined recursively as:
\begin{itemize}
    \item If $M$ has single element, then $\cd(M) = 1$.
    \item If $M$ is not connected, then the contraction-depth of $M$ is the maximum contraction-depth a component of $M$.
    \item If $M$ is connected, then $\cd(M) = 1 + \min_{e \in M} \cd(M/ e)$.
\end{itemize}
\end{definition}

From the fact that deletion and contraction are dual operations one easily proves the following.
\begin{lemma}
\label{lem:cd_dd_dual}
    For any matroid $M$ we have $cd(M) = dd(M^*)$ and $cd(M^*) = dd(M)$.
\end{lemma}

For representable matroids, a generalization of contraction-depth called contraction$^*$-depth was introduced in~\cite{KarKLM17}. The difference between these two notions is that in contraction$^*$-depth we do not have to perform contraction by an element of $M$, but by arbitrary one-dimensional subspace of the space in which the matroid is represented.
\begin{definition}[\cite{ChaCKKP19}]
\label{def:csd}  
For a represented matroid $M = M(A)$, the \emph{contraction$^*$-depth} $\csd(M)$ of $M$ is defined recursively as:
\begin{itemize}
    \item If $M$ has rank $0$, then $\csd(M) = 0$.
    \item If $M$ is not connected, then the contraction$^*$-depth of $M$ is the maximum contraction$^*$-depth a component of $M$.
    \item If $M$ is connected, then $\csd(M) = 1 + \min \csd(M/\ve{v})$, where $\ve{v} \in \F^h$ (here $h$ is the number of rows in $A$).
\end{itemize}
\end{definition}

For a matrix $A$, we define $\csd(A):= \csd(M(A))$.

\begin{theorem}[\cite{ChaCKKP19}]
   Let $A$ be a matrix over a field $\F$. The contraction$^*$-depth of $A$ is the minimum dual tree-depth of any matrix $A'$ that is row equivalent to $A$. 
\end{theorem}

Since $\csd(A) = \csd(M(A))$ and row equivalent matrices describe the same matroids, we have the following corollary.
\begin{corollary}
\label{cor:dual_td}
    If $M$ is a matroid representable over $\F$ with $\csd(M) = d$, then it can be represented by a matrix $A$ over $\F$ of dual tree-depth $d$.
\end{corollary}

We will use the fact that contraction$^*$-depth is closed under deletion of elements.
\begin{lemma}[\cite{KarKLM17}]
\label{lem:restriction}
    For every representable matroid $M$ and any subset $X'$ of elements of $M$ we have $\csd(M[X]) \le \csd(M)$.
\end{lemma}

DeVos, Kwon and Oum~\cite{DeVKO20} established the following relationship between contraction-depth and maximum circuit length:

\begin{theorem}
    \label{thm:cd-circuits}
    Let $M$ be a matroid and $c$ the length of its largest circuit;
    if $M$ has no circuit, then set $c = 1$. 
    Then
    \begin{align*}
        \log_2{(c)} \leq \cd(M) \leq \frac{c(c + 1)}{2}
    \end{align*}
\end{theorem}

Kardoš, Kráľ, Liebenau and Mach~\cite{KarKLM17} showed the following for contraction$^*$-depth:

\begin{proposition}
    \label{thm:csd-circuits}
    Let $M$ be a matroid and $c$ the size of its largest circuit. 
    Then 
    \begin{align*}
        \log_2 c \leq \csd(M) \leq c^2.
    \end{align*}
\end{proposition}

As an easy corollary of the previous two results, we see that contraction-depth and contraction$^*$-depth are functionally equivalent. This was already observed in~\cite{DeVKO20} without a proof, so we provide a proof for completeness.
\begin{theorem}
\label{thm:cd_csd_fequivalent}
    There exists a function $f$ such that $\csd(M) \leq \cd(M) \leq f(\csd(M))$ for any matroid $M$.
\end{theorem}

\begin{proof}
    The first inequality is easily seen from the definitions of contraction-depth and contraction$^*$-depth.
    For the second inequality, let $M$ be a matroid of contraction$^*$-depth $\csd(M) = d$, and let $c$
    denote the length of its largest circuit. 
    By Theorem~\ref{thm:csd-circuits}, $\log_2 c \leq k$, and
    hence $c \leq 2^d$. By Theorem ~\ref{thm:cd-circuits},
    we have $\cd(M) \leq \frac{1}{2}(c(c + 1))$, and by plugging in the previous inequality, we obtain that
    $\cd(M) \leq \frac{1}{2} (2^d (2^d + 1)) = \frac{1}{2} (2^{2\csd(M)} + 2^{\csd(M)})$.
\end{proof}

Finally, to extend some of our results to $\Q$-representable matroids we will need the following lemma.
\begin{lemma}[\cite{ChaCKKP19}]
\label{lem:rational_to_finite}
    There exists a function $h: \N \times \N \to \N$ such that for every matroid $M$ represented over $\mathbb{Q}$ with entry complexity $e$ can be represented over $\F_q$ with $q \le h(e, \csd(M))$.
\end{lemma}

\subsection{Logic}

We will use Monadic Second Order logic with modulo counting (CMSO) in some of our proofs. We will specialize our treatment of logic to the two domains that will be relevant for us -- rooted labelled trees and matrices.

We model rooted labelled trees that use labels from a finite set $S$ as structures over the signature consisting of a binary predicate symbol $par$ and for each $a \in S$ a unary predicate symbol $L_a(x)$. These predicates have the natural meaning -- $par(x,y)$ says that $x$ is a parent of $y$ and $L_a(x)$ means that node $x$ is labelled with label $a$. For example, we can
express that $x$ is a leaf by $\lnot(\exists y. par(x,y))$ and 
express `Every leaf has label $a$' by the sentence $\forall x.(leaf(x)\rightarrow L_a(x))$.


\begin{figure}[H]
    \centering
    \begin{minipage}{0.40\textwidth}

    \begin{align*}
        \begin{bNiceMatrix}[first-col,first-row]
        & \phantom{-} c_1 & c_2 & c_3 & c_4 \phantom{-} \\
        r_1 & \phantom{-} 1 & 1 & 2 & 0 \phantom{-}\\
        r_2 & \phantom{-} 2 & 0 & 1 & 2 \phantom{-}\\
        r_3 & \phantom{-} 1 & 2 & 0 & 1 \phantom{-}\\
        \end{bNiceMatrix}
    \end{align*}

    \end{minipage}
    \begin{minipage}{0.50\textwidth}

    \begin{align*}
        \mathcal{S}(A) &= \{r_1, r_2, r_3, c_1, c_2, c_3, c_4\}\\
        R^{\mathcal{S}(A)} &= \{r_1, r_2, r_3\}\\
        C^{\mathcal{S}(A)} &= \{c_1, c_2, c_3, c_4\}\\
        Entry^{\mathcal{S}(A)}_0 &= \{(r_1, c_4), (r_2, c_2), (r_3, c_3)\}\\
        Entry^{\mathcal{S}(A)}_1 &= \{(r_1, c_1), (r_1, c_2), (r_2, c_3), (r_3, c_1), (r_3, c_4)\}\\
        Entry^{\mathcal{S}(A)}_2 &= \{(r_1, c_3), (r_2, c_1), (r_2, c_4), (r_3, c_2)\}
    \end{align*}

    \end{minipage}
    \caption{An example of a matrix $A$ and its associated structure $\mathcal{S}(A)$.}
    \label{fig:structure_example}
\end{figure}

For a finite field $\F$ we model matrices with entries from $\F$ as structures over the signature $\Sigma_\F = \{R,C\} \cup \{Entry_\alpha\}_{\alpha \in \F}$ where $R$ and $C$ are unary predicate symbols and each $Entry_{\alpha}$ is a binary predicate symbol.
To any matrix $A$ over $\F$, we associate a  $\Sigma_\F$-structure $\S(A)$ as follows: The universe if $\S(A)$ is $\rows(A) \cup \columns(A)$. Unary predicate $R$ is realized as $R^{\S(A)} = \rows(A)$ and analogously $C^{\S(A)} = \columns(A)$.
To define the realization of $Entry_\alpha$, for a row $r$ and a column $c$ we denote by $A(r,c)$ the entry of $A$ in row $r$ and column $c$.
Then, for each $\alpha \in \F$ we define 
$$ Entry^{\S(A)}_\alpha = \{ (r,c)~|~r \in \rows(A), c \in \columns(A), A(r,c) = \alpha \} $$

For example, if we want to say that every row of $A$ contains a specific element $\alpha \in \F$, we can express it as follows: $\forall x (R(x) \rightarrow \exists y(C(y) \land Entry_\alpha(x,y))$. To shorten and unclutter our formulas, we will introduce row and column quantifiers: $\forall_R x$ means `for each row $x$' and $\exists_R x$ means `there exists a row $x$'. For columns the quantifiers $\forall_C$ and $\exists_C$ are defined analogously. The example formula given above can then be expressed as $\forall_R x \exists_C y. Entry_\alpha(x,y)$. 

The examples of formulas given above were formulated in the first-order logic (FO), meaning that the quantification was over the elements of the universe in question (nodes of a tree or rows/columns of a matrix). In MSO logic we can also quantify over subsets of the universe. In this setting, we use capital letters $X,Y,\ldots$ for set variables, and we use predicate $x \in X$ to denote that $x$ belongs to set $X$. For example, the sentence $\exists X. \forall x.((leaf(x) \rightarrow x \in X)\land (root(x) \rightarrow x \not\in X)) \land  (\forall x \forall y (par(x,y) \rightarrow (x \in X \leftrightarrow y \not\in X))$ expresses that all leaves are at odd distance from the root.  
For formulas in the language of matrices, we can consider set quantifiers that speak of sets of rows or columns only, so for example $\forall_C X$ reads as `for all sets of rows' and the quantifiers $\exists_C X$, $\forall_R X$, $\exists_R X$ are defined analogously. We note that these quantifiers are only syntactic shorthands, and they can be replaced by the standard set quantifiers. For example, the formula $\forall X_C. \phi(X)$ can be expressed as $\forall X. (\forall x. (x \in X \rightarrow C(x))) \rightarrow \phi(X))$.

Finally, in the CMSO logic we can also use predicates of the form  $mod_{a,b}(X)$; the semantics of this predicate is that $|X| = a~(\text{mod } b)$.

An \emph{interpretation} of matrices in trees is a tuple $I = (\nu(x), \rho_C(x), \rho_R(x), \{\psi_\alpha(x,y)\}_{\alpha \in \F})$ of formulas over the signature of labelled trees. To labelled tree $T$ the interpretation $I$ assigns a matrix structure $I(T)$ as follows: The universe of $I(T)$ is the set $\{v\in V(T)~|~T \models \nu(v)\}$. The realization of predicate $R$ in $I(T)$ is given by $\{v\in V(T)~|~T \models \rho_R(v)\}$, and the realization of predicate $C$ is defined analogously. Finally, for each $\alpha \in \F$ the predicate $Entry_\alpha$ is realized as $\{(u,v)~|~T \models \psi_{\alpha}(u,v)\}$.

To a sentence $\varphi$ in the language of matrices, the interpretation $I$ assigns a new sentence $I(\varphi)$ in the language of trees as follows. Each occurrence of $R(x)$ and $C(x)$ is replaced by $\rho_R(x)$ and $\rho_C(x)$, respectively, and each occurrence of $Entry_\alpha(x,y)$ is replaced by $\psi_\alpha(x,y)$. Moreover, all quantifications are relativized to elements $x$ that satisfy $\nu(x)$.

We will need the following lemma, proof of which (in the case of general interpretations) can be found for example in~\cite{Hod97}.
\begin{lemma}
\label{lem:interp}
    Let $\phi$ be a sentence over the language of matrices and let $I$ be an interpretation of matrices in trees. Then for every tree $T$ we have 
    $I(T) \models \phi \Longleftrightarrow T \models I(\phi)$.
\end{lemma}




\section{Expressing contraction$^*$-depth in CMSO}

In this section, we show that in CMSO it is possible to express that a matrix $A$ represents a matroid of contraction$^*$-depth at most $d$. This result will be crucial in the proof of Theorem~\ref{thm:main1}.

\begin{theorem}
\label{thm:csd_CMSO}
    For every $d$ there exists a CMSO sentence $csd_d$ such that for every matrix $A$ over $\F$ we have $\S_A \models csd_d$ if and only if $\csd(M(A)) \le d$.
\end{theorem}

To prove Theorem~\ref{thm:csd_CMSO} we need to overcome several difficulties:
\begin{itemize}
    \item The definition of contraction$^*$-depth is phrased in terms of matroids, but the sentence $csd_d$ from Theorem~\ref{thm:csd_CMSO} has to work on a matrix (on a matrix structure, to be more precise). We will therefore rephrase Definition~\ref{def:csd} in terms of matrices and quotient spaces (see Section~\ref{sec:matrices}). 
    \item Let $A$ be a matrix over $\F$ and let $m$ be the number of rows of $A$. The definition of contraction$^*$-depth uses an operation where we pick an arbitrary vector $\vev$ in $\FF^m$ and we quotient the vector space we are currently working with by $\vev$. Expressing this in CMSO would be easy if $\vev$ was one of the columns of $A$, but in general we cannot assume this, and so our formula $csd_d$ has to be able to quantify over all vectors from $\FF^m$. We resolve this by introducing \emph{virtual columns} (Section~\ref{sec:virtual_columns}).
    \item In our formula $csd_d$, we need to be able to check whether a set $X$ of columns of $A$ is linearly (in)dependent over $\F$ (Section~\ref{sec:dep}). To do this, our formula will have to be able to speak about sums of arbitrarily large number of entries that are in the same row of $A$. We will now illustrate how we can address this on a simpler problem -- checking whether all entries of $A$ in a given row $r$ of $A$ sum to $0$.  To check this, we can proceed as follows. Let $p:=|\F|$ and let $\alpha_1,\ldots, \alpha_p$ be an enumeration of elements of $\F$. We partition the columns of $A$ into sets $X_1,\ldots, X_p$ by putting $c$ into $X_i$ if $A(r,c) = \alpha_i$. For each $i \in [p]$ set $k_i := |X_i|$. Then our problem reduces to checking whether $k_1\alpha_1 + \ldots + k_p \alpha_p =0$. While each $k_i$ can be arbitrarily large, we know that for each $\alpha \in \F$ we have $p\alpha = 0$ in $\F$, and so $k_i \alpha_i = r_i \alpha_i$, where $r_i$ is the reminder of $k_i/p$, i.e. $r_i = k_i (\text{mod } p)$. Thus, we only need to check whether $r_1\alpha_1 + \ldots + r_p \alpha_p =0$. This is where we use the modulo counting feature of CMSO: we can use predicates of the form $mod_{r_i,p}(X_i)$ to determine $r_i$.
    Since for us the value $p$ is a fixed constant (our formula is allowed to depend on $\F$) and for each $i$ we have $r_i < p$, the formula that checks whether the entries in row $r$ sum to $0$ can be written as
    $$ ZeroSum(x):=\exists X_1 \ldots \exists X_p~\forall y \left( \bigwedge_{i=1}^p (c \in X_i \leftrightarrow Entry_{\alpha_i}(x,y))\right) \land $$
    $$ \land \bigvee_{\substack{r_1,\ldots,r_p \in \F \\ r_1\alpha_1 + \ldots + r_p\alpha_p = 0}} \left( \bigwedge_{i=1}^p mod_{r_i,p}(X_i) \right). $$
\end{itemize}

\subsection{Rephrasing the definition of contraction$^*$-depth }
\label{sec:matrices}

To be able to speak in CMSO about contraction$^*$-depth of a matrix, it will be convenient for us to rephrase the definition contraction$^*$-depth explicitly in terms of quotient spaces.


Let $A$ be a matrix $A$ and let $\ve {v_1},\ldots, \ve {v_k}$ be vectors from $\F^m$, where $m$ is the number of rows of $A$. 
We say that a set $X = \{\ve{u_1},\ldots, \ve{u_\ell}\}$ of columns of $A$ (where we treat them as vectors from $\F^m$) is \emph{dependent with respect to $\ve {v_1},\ldots, \ve {v_k}$} if there exist coefficients $\alpha_1, \ldots, \alpha_\ell \in \F$, not all $0$, such that $\alpha_1  \ve{u_1} + \ldots + \alpha_\ell \ \ve{u_\ell} \in \Span(\ve {v_1},\ldots, \ve {v_k})$. In other words, $X$ is dependent with respect to $\ve {v_1},\ldots, \ve {v_k}$ if it is dependent in $\F^m/\Span(\ve {v_1},\ldots, \ve {v_k})$.

Based on $A$ and $\ve {v_1},\ldots, \ve {v_k}$ we define matroid $M(A,\{\ve {v_1},\ldots, \ve {v_k}\})$ as follows. The ground set of $M(A,\{\ve {v_1},\ldots, \ve {v_k}\})$ is the set of all columns of $A$, and we say that a subset $X = \{u_1,\ldots, u_s\}$ of $\columns(A)$ is dependent in $M(A,\{\ve {v_1},\ldots, \ve {v_k} \})$ if it is dependent with respect to $\ve {v_1},\ldots, \ve {v_k}$.

The contraction$^*$-depth $\csd(A, \{\ve {v_1},\ldots, \ve {v_k}\})$ of a matrix $A$ and set $\{\ve {v_1},\ldots, \ve {v_k}\}$ of columns is defined recursively as:
\begin{itemize}
    \item If $M(A,\{\ve {v_1},\ldots, \ve {v_k}\})$ has rank 0, then $\csd(A,\{\ve {v_1},\ldots, \ve {v_k}\}) = 0$.
    \item If $M(A,\{\ve {v_1},\ldots, \ve {v_k}\})$ is not connected, then the contraction$^*$-depth of $A,\{\ve {v_1},\ldots, \ve {v_k}\}$ is the maximum contraction$^*$-depth of any $A',\{\ve {v_1},\ldots, \ve {v_k}\}$ where $A'$ is a submatrix of $A$ formed by the elements of a connected component of $M(A,\{\ve {v_1},\ldots, \ve {v_k}\})$.
    \item If $M$ is connected, then $\csd(A,\{\ve {v_1},\ldots, \ve {v_k}\}) = 1 + \min \csd(A,\{\ve {v_1},\ldots, \ve {v_k}, \ve {v_{k+1}}\})$, where the minimum is taken over all vectors $\ve {v_{k+1}}$ from $\F^m$.
\end{itemize}
It is easily verified that for a matrix $A$ we have $\csd(M(A)) = \csd(A,\emptyset)$.

\subsection{Virtual columns}
\label{sec:virtual_columns}

Let $A$ be a matrix and $\S_A$ the corresponding matrix structure. A \emph{virtual column} of $\S(A)$ is an
$\F$-indexed family $\{Q_\alpha\}_{\alpha \in \F}$ of subsets of rows of $\S_A$ such that $\bigcup_{\alpha \in \F} Q_\alpha = R$ (recall that $R$ is the set of elements  of $\S_A$ that correspond to rows of $A$) and $Q_\alpha \cap Q_\beta = \emptyset$ if $\alpha \not= \beta$. To simplify the notation, we will denote virtual columns by $\langle Q \rangle$ instead of $\{Q_\alpha\}_{\alpha \in \F}$.

Let $R = \{r_1,\ldots, r_m\}$ be the set of rows of $A$. To every virtual column $\vc{Q}$  we assign a vector $\vev(\vc{Q}) = (v_1,\ldots, v_m)^T \in \F^m$ by setting $v_i := \alpha$ if $r_i \in Q_\alpha$.

Since every virtual column corresponds to a finite family of subsets of rows of $\S_A$, we can work with virtual columns in CMSO -- they correspond to $\F$-indexed families of set variables. Again, to simplify the notation, we will write $\langle X \rangle$ instead of $\{X_\alpha\}_{\alpha \in \F}$.
Note that we can quantify over virtual columns --
for example, $\exists \vc{X}$ can be implemented as $\exists_C X_{\alpha_1} \ldots \exists_C X_{\alpha_p}$, where $\alpha_1,\ldots, 
\alpha_p$ is an enumeration of elements of $\F$. Universal quantification is handled analogously. 
We will need the following lemma.

\begin{lemma}
\label{lem:LinCom}
For every $k$, there exists a formula $LinCom(\la Z \ra,\la Y^1 \ra ,\ldots, \la Y^k \ra)$ such that for any matrix structure $S(A)$ and any virtual columns $\la L \ra, \la J^1 \ra,\ldots, \la J^k \ra $ we have that 
$\S(A) \models LinCom(\vc{L},\vc{J^1},\ldots, \vc{J^k})$ if and only if $\vev(\vc{L}) \in \Span(\vev(\vc{J^1}),\ldots, \vev(\vc{J^k}))$.
\end{lemma}

\begin{proof}
    For any $(\alpha_1,\ldots, \alpha_k) \in \F^{k}$ define 
 $$I(\alpha_1,\ldots, \alpha_k):= \{(\gamma,\beta_1,\ldots, \beta_k)~|~\alpha_1\beta_1 + \ldots + \alpha_k \beta_k = \gamma \}$$

We define the formula $LinCom(\la Z \ra,\la Y^1 \ra ,\ldots, \la Y^k \ra)$ by 
 $$ \bigvee_{(\alpha_1,\ldots, \alpha_k) \in \F^{k}} \forall x. R(x) \rightarrow \left( \bigvee_{(\gamma,\beta_1,\ldots, \beta_k) \in I(\alpha_1,\ldots, \alpha_k)} \left (\bigwedge_{i=1}^k (x \in Y^i_{\beta_i}) 
  \land x \in Z_\gamma \right) \right) $$

We now argue that the formula has the required property. Let $\la L \ra, \la J^1 \ra,\ldots, \la J^k \ra $ be virtual columns and set $\veu:=\vev(\vc{L})$, $\ve {v_1}:= \vev(\vc{J^1}), \ldots, \ve {v_k}:= \vev(\vc{J^k})$ to simplify the notation. Then $u \in \Span(\ve {v_1},\ldots, \ve {v_k})$ if and only if there exist $\alpha_1,\ldots, \alpha_k$ (the first disjunction in our formula) such that $\alpha_1 \ve {v_1} + \ldots + \alpha_k \ve {v_k} = \veu$. This equality is true if and only if for every row $x$ of $\S_A$ we have that if $\gamma$ is the value of $Z$ on row $r$ and $\beta_i$ is the value of $Y^i$ on row $r$, then $\gamma = \alpha_1 \beta_1 + \ldots + \alpha_k \beta_k$ (this is the innermost conjunction in our formula; the big disjunction over members of $I(\alpha_1,\ldots,\alpha_k)$ is there to obtain the values of all entries of $\la Z \ra,\la Y^1 \ra ,\ldots, \la Y^k \ra$ at row $x$). 
\end{proof}

\subsection{Linear dependence in quotient spaces in CMSO}
\label{sec:dep}
\begin{lemma}
    For every $k$ there exists a CMSO formula $Dep(X, \vc{Y^1}, \ldots, \vc{Y^k})$ such that for any matrix $A$ over $\F$, any set $S$ of columns of $A$ and any virtual columns $\vc{Q^1}, \ldots, \vc{Q^k}$ we have that $S$ is linearly dependent with respect to $\vev(\vc{Q^1}), \ldots, \vev(\vc{Q^k})$ if and only if $\S(A) \models Dep(S, \vc{Q^1}, \ldots, \vc{Q^k})$.
\end{lemma}
\begin{proof}
It suffices to show that there exists a formula $Dep(X, \vc{Z})$ that expresses that the set $X$ of columns is dependent with respect to $\vc{Z}$. Then, the desired formula $Dep(X, \vc{Y^1}, \ldots, \vc{Y^k})$ can be written as
$$ Dep(X, \vc{Y^1}, \ldots, \vc{Y^k}):= \exists \vc{Z} \left( Dep(X, \vc{Z}) \land LinCom(\vc{Z}, \vc{Y_1}, \ldots, \vc{Y_k} ) \right) $$
where $LinCom$ is the formula from Lemma~\ref{lem:LinCom}.

Next, set $p := |\F|$. We now proceed to building the formula $Dep(X, \vc{Z})$ in steps from simpler formulas.
First, we show that for every $\alpha, \beta \in \F$ there exists a formula $sum_{\alpha,\beta}(X,y)$ such that the following holds: If $S$ is a set of columns of $A$ and $r$ is a row of $A$, then $\S_A \models sum_{\alpha,\beta}(S,r)$ if and only if 
$$ \sum_{\substack{c \in S \\ A(r,c) = \beta}} A(r,c) = \alpha. $$
In words, the formula $sum_{\alpha,\beta}(X,y)$ checks whether $\alpha$ is the sum of all entries of $A$ in row $r$ that intersect $Q$ and  have value $\beta$. We now describe the formula's construction. Let $S'$ be the subset of $S$ that contains precisely the columns $c$ for which $A(r,c) = \beta$. 
Note that all the entries $A(r,c)$ in the sum above are equal to $\beta$, and so we have 
$$ \sum_{\substack{c \in S \\ A(r,c) = \beta}} A(r,c) = \sum_{c\in S'} A(r,c) =  |S'|\cdot \beta $$
Note that in $\F$ we have $|S'|\cdot \beta = q\cdot\beta$, where $q \equiv |S'|\mod p$ (this is because $p\cdot \beta = 0$), and so the formula $sum_{\alpha,\beta}$ only needs to check whether $|S'| \equiv q \mod p$ for some $q$ with $q\cdot \beta = \alpha$. Thus, we can write $sum_{\alpha,\beta}$  as follows:
$$ sum_{\alpha,\beta}(X,y):= \bigvee_q \exists_C Z [(\forall z \in Z \leftrightarrow (Entry_{\beta}(y,z) \land z \in X)) \land 
mod_{q,p}(Z) ]$$ 
where the disjunction is over all $q \le p $ with $q\cdot \beta = \alpha$.

Next, we describe, for each $\gamma \in \F$, the formula $sum_\gamma(Y,y)$ such that for every row $r$ and every set $S$ of columns of $A$ we have that  $\S(A) \models sum_\gamma(S,r)$ if and only if 
\begin{equation}
\label{eq:sum_alpha}
\sum_{c \in Q } A(r,c) = \gamma. 
\end{equation}
Let us enumerate the elements of $\F$ as $\beta_1, \ldots, \beta_p$. Clearly we have
$$ \sum_{c \in S } A(r,c) = \sum_{i=1}^p  \sum_{\substack{c \in S \\ A(r,c) = \beta_i}} A(r,c).$$
Consequently, we have that~(\ref{eq:sum_alpha}) holds if and only if for some $\alpha_1,\ldots, \alpha_p$ with $\alpha_1 + \ldots + \alpha_p = \gamma$ we have
$$ \S(A) \models \bigwedge_{i=1}^p sum_{\alpha_i, \beta_i} (S,r) $$
Thus, we can write the formula $sum_\gamma(Y,y)$ as
$$ sum_\gamma(Y,y):= \bigvee_{\substack{(\alpha_1, \ldots, \alpha_p) \\\alpha_1 + \ldots + \alpha_p = \gamma}} \left(  \bigwedge_{i=1}^p sum_{\alpha_i, \beta_i} (Y,y) \right) $$

Finally, we describe the formula $Dep(X, \vc{Z})$ such that for every set $S$ of columns and every virtual column $\vc{Q}$ of $A$ we have that $\S(A) \models Dep(S, \vc{Q})$ if and only if $S$ is dependent with respect to $\vev(\vc{Q})$. By definition, $S = \{u_1,\ldots, u_\ell\}$ is dependent with respect to $v(\vc{Q})$ if there are coefficients $\alpha_1,\ldots,\alpha_\ell$ in $\F$, not all $0$, such that $ \alpha_1 u_1 + \ldots + \alpha_\ell u_\ell = \delta v(\vc{Q})$ for some $\delta \in \F$. In the language of $\S(A)$, where we have $S = \{c_1,\ldots, c_\ell\}$, this translates to requiring that for every row $r$ we have 
\begin{equation}
\label{eq:blah}
     \sum_{j=1}^\ell \alpha_i A(r,c_i) = \delta \sigma_r, 
\end{equation}
where $\sigma_r$ is the entry of $v(\vc{z})$ at row $r$.

Again, let $\beta_1,\ldots, \beta_p$ be an enumeration of the elements of $\F$. Since each coefficient $\alpha_j$ is equal to some $\beta_i$, we can think of an assignment of coefficients $\alpha_1,\ldots,\alpha_\ell$ to columns in $S$ as of partitioning of $S$ into sets $Y_1,\ldots,Y_p$, where we put $\alpha_j$ into $Y_i$ if $\alpha_j = \beta_i$. Here we require at least one $Y_i$ with $\beta_i\not=0$ is non-empty. Condition~(\ref{eq:blah}) then says that for each row $r$ we require that 
$$ \sum_{i=1}^p (\beta_i \cdot \sum_{c \in Y_i} A(r,c)) = \sigma_r$$

This happens if for each row $r$ of $A$ there are $\gamma_1,\ldots, \gamma_p$ such that $\beta_1 \cdot \gamma_1 + \ldots + \beta_p \cdot \gamma_p = \sigma_r$ and  for each $i \in [p]$ we have that $ \S_A \models sum_{\gamma_i}(Y_i,r)$. Thus, we can write $Dep(X,\vc{Z})$ as
$$ \pushQED{\qed}
Dep(X, \vc{Z}):= \exists Y_1 \ldots \exists Y_p \forall_R y \bigvee_{\substack{(\gamma_1, \ldots, \gamma_p, \sigma) \\ \beta_1 \cdot \gamma_1 + \ldots + \beta_p \cdot \gamma_p = \sigma}} \left( \bigwedge_{1 \le i \le p} sum_{\gamma_i}(Y_i,c) \land y \in Z_{\sigma}
\right) \land \bigvee_{\substack{i=1 \\ \beta_i\not=0}}^p Y_i\not=\emptyset.
\qedhere
\popQED
$$
\renewcommand{\qedsymbol}{}
\end{proof}

\subsection{Formula for contraction$^*$-depth}

In this section we will prove Theorem~\ref{thm:csd_CMSO}. We start by showing that we can express connectivity and connected components in CMSO. To unclutter the notation we will assume that all set quantifiers in this section are over sets of columns, and we will write $\exists$, $\forall$ everywhere instead of $\exists_C$ and $\forall_C$.
\begin{lemma}
    For every $k$ there is a CMSO formula $cc(X,Z, \vc{Y_1}, \ldots, \vc{Y_k})$ such that for every matrix $A$, subsets $S$, $T$ of columns of $A$ with $S \subseteq T$ and virtual columns $\vc{Q_1},\ldots, \vc{Q_k}$ we have that $\S(A) \models cc(S,T,\vc{Q_1},\ldots, \vc{Q_k})$ if and only if $S$ is a connected component of $M(A,v(\vc{Q_k}),\ldots,v(\vc{Q_k}))[T]$.
\end{lemma}
\begin{proof}
    We can express that a set $Y$ of columns is a circuit: 
    $$ circ(Y,  \vc{Y_1}, \ldots, \vc{Y_k}):= Dep(Y,  \vc{Y_1}, \ldots, \vc{Y_k}) \land (\forall Y'. (Y' \subsetneq Y \rightarrow \lnot Dep(Y',  \vc{Y_1}, \ldots, \vc{Y_k}))  $$
Using this, it is straightforward to express that a set $X$ of columns is connected in the restriction of the original matroid to a set $Z$ of columns:
 \begin{align*}
 conn(X, Z, \vc{Y_1}, \ldots, \vc{Y_k}):= & \forall x \forall y (x \in X \land y \in X) \rightarrow \\
 & \rightarrow (\exists Y. Y \subseteq Z \land x \in Y \land y \in Y \land circ(Y, \vc{Y_1}, \ldots, \vc{Y_k})) 
 \end{align*}

Finally, it is easy to express that $X$ is a maximal connected set in $Z$.
\begin{align*}
 cc(X, Z, \vc{Y_1}, \ldots, \vc{Y_k}):=& conn(X, Z,\vc{Y_1}, \ldots, \vc{Y_k}) \land \\ &\land \forall Y. \left( (X \subseteq Z \land X \subsetneq Y) \rightarrow \lnot conn(Y, Z, \vc{Y_1}, \ldots, \vc{Y_k}) \right)     
\end{align*}

\end{proof}

\begin{proof}[Proof of Theorem~\ref{thm:csd_CMSO}]
For every $d$ we provide a formula that expresses that a subset $Z$ of columns of a matrix $A$ represents a matroid of contraction$^*$-depth at most $d$.

For contraction$^*$-depth $0$, we just need to check whether all non-empty subsets of columns are dependent (meaning that the only independent set has size $0$ and so the rank is $0$):
$$csd_0(Z,  \vc{Y_1}, \ldots, \vc{Y_k}):= \forall X. \left( X \subseteq Z \land X \not= \emptyset \right) \rightarrow Dep(X,  \vc{Y_1}, \ldots, \vc{Y_k})$$
For contraction$^*$-depth $d > 1$, we proceed according to the recursive definition:
\begin{align*}
    csd_d&(Z,\vc{Y_1},\ldots,\vc{Y_k}):=(conn(Z,Z, \vc{Y_1},\ldots,\vc{Y_k}) \rightarrow \exists \vc{Y}.(csd_{d-1}(Z,\vc{Y_1},\ldots,\vc{Y_k},\vc{Y})) \\ 
    & \land  (\lnot conn(Z,Z \vc{Y_1},\ldots,\vc{Y_k}) \rightarrow (\exists X. cc(X,Z, \vc{Y_1},\ldots,\vc{Y_k}) \rightarrow csd_{d}(X, \vc{Y_1},\ldots,\vc{Y_l}))) \\
    & \land \forall X. cc(X,Z, \vc{Y_1},\ldots,\vc{Y_k}) \rightarrow \bigvee_{i=0}^d csd_{i}(X, \vc{Y_1},\ldots,\vc{Y_l})))
\end{align*}

Finally, the sentence $csd_d$ will work on the set of all column vectors and will not use any  virtual columns:
$$ \pushQED{\qed} 
csd_d:= \exists Z \left((\forall x (x \in Z))  \land csd_d(Z)\right) 
\qedhere
\popQED
$$ 
\renewcommand{\qedsymbol}{}
\end{proof}

\section{Minimal matroids of given contraction$^*$-depth}

In this section we prove Theorem~\ref{thm:main1}, showing that deletion-minimal matroids with contraction$^*$-depth $d$ have size bounded in terms of $d$ and $|\F|$. We will proceed as follows. 
\begin{itemize}
    \item Going for contradiction, we will assume that there exists a matroid $M$ with $\csd(M) = d$ and $|M| > N$, where $N$ is suitably chosen number depending on $d$ and $|\FF|$. 
    \item Based on Corollary~\ref{cor:dual_td}, $M$ can be represented by a matrix $A$ of dual tree-depth $d$. We will show each such matrix can be encoded in a colored tree $T$ of depth $d+2$ in such a way that we can recover matrix the structure $\S(A)$ from $T$ by a CMSO interpretation.
    \item Since $\S(A)$ can be intrepreted in $T$, by Lemma~\ref{lem:interp} we can translate the sentence $csd_d$ from Theorem~\ref{thm:csd_CMSO} to a sentence $csd_d^{tree}$ such that $T \models \csd_d^{tree}$ if and only if the matrix represented by $T$ has contraction$^*$-depth $d$.
    \item By a result of~\cite{GajH15}, if $T$ is larger than $N$ (for some $N$ depending on $d$ and $csd_d^{tree}$), then there is a subtree $T'$ such that $T \models csd_d^{tree}$ if and only if $T' \models csd_d^{tree}$. From this we deduce that the matrix $A'$ that is encoded by $T'$ represents a matroid $M' = M(A')$ of contraction$^*$-depth $d$. We will show that this matroid is a restriction of $M$, and this will yield a contradiction with our assumption that $M$ is deletion-minimal matroid of contraction$^*$-depth $d$.
\end{itemize}

\subsection{Encoding matrices in trees}
\label{sec:encoding_matrices}

A \emph{$(d, \F)$-matrix-tree} is a rooted tree $T$ of depth at most $d$ equipped with unary predicates from the set $\{L_R, L_C\} \cup \{L_{i,\alpha}~|~ i \in \{0,\ldots, d\} \text{ and } \alpha \in \FF\}$ with the following properties:
\begin{itemize}
    \item If a node $v$ of $T$ is marked with a unary predicate $L_C$, then $v$ is a leaf.
    \item All other nodes of $T$ except for the root are marked with $L_R$.
    \item If $v$ is a leaf marked with $L_C$ and $v$ is at distance $k \le d$ from the root, then for each $i \in \{0,\ldots, k-1\}$ there exists exactly one $\alpha \in \F$ such that $v$ is marked with $L_{i,\alpha}$.
\end{itemize}

To each $(d, \F)$-matrix-tree  $T$ we assign a matrix $A_T$ as follows:
\begin{itemize}
    \item The rows of $A_T$ are the nodes of $T$ marked with $L_R$.
    \item The columns of $A_T$ are the nodes of $T$ marked with $L_C$.
    \item The entry $A_T$ at row $r$ and column $c$ is determined by the following rule: If $r$ is not an ancestor of $c$ in $T$, then $A(T)(r,c) := 0$. If $r$ is an ancestor of $c$ in $T$ at distance $i$ from the root, then $A(T)(r,c) := \alpha$, where $\alpha$ is the unique element of $\F$ such that $c$ is marked with $L_{i,\alpha}$.
\end{itemize}

\begin{lemma}
\label{lemma:matrix-tree}
    Let $A$ be a matrix over $\F$ of dual tree-depth at most $d$. Then there is a $(d+3,\F)$-matrix-tree $T(A)$ such that $A = A(T(A))$.
\end{lemma}
\begin{proof}
We construct the tree as follows:
    \begin{itemize}
        \item Let $F$ be an elimination forest of the dual graph $G_D(A)$ of $A$ of depth $d$ + 1; note that vertices of $F$ are the rows of $A$. For every column $c$ of $A$ do the following: Let $S_c$ be the set of rows $r$ for which $A(r,c) \not= 0$. By construction of $G_D(A)$, these rows form a clique $K$ in $G_D(A)$, and so by Lemma~\ref{lem:clique} there is a vertex $v \in S_c$ such that all vertices in $S_c$ are ancestors of $v$ in $F$. Create a new vertex for column $c$ and make it a child of $v$. Finally, create a new vertex and make it a parent of the roots of connected components of $F$. This finishes the construction of $T(A)$, apart from the labelling.
        \item We now describe the labelling of $T(A)$. All vertices of the original elimination forest $F$ will be marked with unary predicate $L_R$ and all newly created vertices for columns will receive unary predicate $L_C$. Finally, we will mark every vertex $c$ corresponding to a column of $A$ as follows: if $c$ has an ancestor $r$ at distance $i$ from the root of $T$ and the entry $A(r,c) = \alpha$, then we mark $c$ with label $L_{i,\alpha}$.
    \end{itemize}
The fact that $A = A(T(A))$ follows directly from the construction.
\end{proof}

\begin{lemma}
\label{lemma:interp_matrix_tree}
There is an interpretation $I$ from the language of labelled rooted trees to the language matrices over $\F$ such that 
for any matrix $A$ of dual tree-depth $d$ we have that $\S(A) = I(T(A))$.
\end{lemma}
\begin{proof}
To show this, we need to describe formulas $\nu(x)$, 
 $\rho(x)_R$, $\rho_C(x)$ and $\psi_{\alpha}(x,y)$ 
 for each $\alpha \in \F$ that define the new universe and express that $x$ is a row of $\S(A)$, $x$ is a column of $\S(A)$ and that the entry of $\S(A)$ at row $x$ and column $y$ is $\alpha$, respectively.
 
 These formulas can be obtained as follows. We set $\nu(x):= \lnot root(x)$, where $root(x)$ is a formula that states that $x$ has no parent. We then set $\rho_R(x):=L_R(x)$ and $\rho_C(x):=L_C(x)$. To recover the entries of $A$, for $\alpha \not= 0$ we set $$\psi_\alpha(x,y):= L_R(x) \land L_C(y) \land anc(x,y) \land \left(\bigvee_{k=1}^d (depth_k(y) \land L^{k,\alpha}(x)) \right)$$
where $anc(x,t)$ says that $y$ is an ancestor of $x$ and $depth_k(x)$ says that $x$ is at distance $k$ from the root; both of these are easily expressible in CMSO. For $\alpha = 0$ we set $\varphi_\alpha(x,y):= L_R(x) \land L_C(y) \land 
 \bigwedge_{\beta\not=0} \lnot \varphi_\beta(x,y)$.
\end{proof}

As a simple corollary of the previous two lemmas we get a CMSO sentence that can check whether contraction$^*$-depth of a matrix $A$ is equal to $d$ by looking at tree $T(A)$ from Lemma~\ref{lemma:matrix-tree}.
\begin{lemma}
\label{lem:csd_tree}
    For every $d$ and $\F$ there exists a CMSO sentence $csd_d^{tree}$ such that if $A$ is a matrix of dual tree-depth $d$ and $T(A)$ is the $(d+3,\F)$-matrix-tree from Lemma~\ref{lemma:matrix-tree}, then $T(A) \models \csd_d^{tree}$ if and only if contraction$^*$-depth of $A$ is $d$.
\end{lemma}
\begin{proof}
    From Theorem~\ref{thm:csd_CMSO} we have a CMSO sentence  $csd_d$ such that $\S(A) \models csd_d$ if and only if $\csd(A) = d$. Since by Lemma~\ref{lemma:interp_matrix_tree} there exists intrepretation $I$ such that $\S(A) = I(T(A))$, by Lemma~\ref{lem:interp} there is a formula $csd_d^{tree}:=I(csd_d)$ such that $\S(A) \models csd_d$ if and only if $T(A) \models csd_d^{tree}$.
\end{proof}

\subsection{Finishing the proof}
\label{sec:proof_finish}

Let $T$ be a tree and let $v$ be a vertex of $T$. A \emph{subtree $T_v$ of $T$ based in $v$} is the subtree of $T$ consisting $v$ together with all its descendants. 
If $v_1,\ldots, v_\ell$ are nodes of $T$, then by $T - (T_{v_1} \cup \ldots \cup T_{v_{\ell}})$ we mean the tree obtained from $T$ by removing all vertices from $T_{v_1} \cup \ldots \cup T_{v_{\ell}}$. 
We will need the following Theorem from~\cite{GajH15}, which we present here in a form adjusted to our purposes.

\begin{theorem}[\cite{GajH15}, Theorem 3.5]
\label{thm:reduce}
    For every $d$, $m$ and every CMSO sentence $\varphi$ there exist $N$ and $M$ with the following property: Let $T$ be a tree of depth at most $d$ labelled with at most $m$ unary predicates such that $T \models \varphi$. If $|T| > N$, then there is a node $v \in V(T)$ and $M$ nodes  $v_1,\ldots, v_M$ among its children such that for $T':=T-(T_{v_1} \cup \ldots \cup T_{v_M})$ we have that $T' \models \varphi$. 
\end{theorem}

Finally, we will need one more technical lemma.
\begin{lemma}
\label{lem:ignore_rows}
Let $T$ be a $(d,\F)$-matrix-tree and let $v$ be its vertex.
Let $T' = T - T_v$, and let $T''$ be obtained from $T$ by removing all leaves of $T_v$ marked with unary predicate $L_c$.
Then $M(A(T'))$ is isomorphic to $M(A(T''))$.
\end{lemma}
\begin{proof}
Set $A' := A(T')$ and $A'' := A(T'')$.
Note that $A'$ and $A''$ have the same set of columns; the difference between them is that $A''$ contains more rows -- these correspond precisely to the vertices of $T_v$ marked with predicate $L_R$.
We will argue that in $A'$ these rows are identically $0$. Then clearly we will have that any set $X$ of columns of $A'$ is independent if an only if it is independent in $A''$, and the result follows.

To argue that all rows contained in $T_v$ are identically $0$ in $A''$, we proceed as follows.
Let $r$ be any row contained in $T_v$.
If $c$ is any column such that $A(r,c)\not=0$, then by construction of $T(A)$ we have that $c$ must be a descendant of $r$, and therefore $c$ must be a leaf of $T_v$. Since $A''$ does not contain any leaves of $T(A)_v$, this finishes the proof.
\end{proof}

Theorem~\ref{thm:main1} will be an easy corollary of the following lemma.
\begin{lemma}
\label{lem:main1}
Let $\F$ be a finite field.
Let $M$ be an $\mathbb{F}$-represented matroid such that
$\csd (M) = d$ and $\csd (M \setminus S) < d$ for every non-empty subset $S$ of elements of $M$. Then $|M| \leq f(|\FF|, d)$ for some function $f$.
\end{lemma}

\begin{proof}
Let $p:=|\F|$ and let $csd_d^{tree}$ be the CMSO sentence from Lemma~\ref{lem:csd_tree}.
Let $N$ be the number obtained by applying Theorem~\ref{thm:reduce} to $d+3$, $m:=dp+2$ and sentence $csd^{tree}_d$; note that $m$ and sentence $csd^{tree}_d$ depend only on $d$ and $\F$.
Assume for contradiction that there is a $\F$-represented matroid as in the assumptions of the lemma for which $|M| > N$. By Corollary~\ref{cor:dual_td} there exists a matrix $A$ of dual tree-depth at most $d$ such that $M(A)$ is isomorphic to $M$. 
By Lemma~\ref{lem:csd_tree} we have that $T(A) \models csd_d^{tree}$. 
We set $T:=T(A)$ for the rest of the proof and we note that $A = A(T)$ by Lemma~\ref{lemma:matrix-tree}.
Since $T$ has depth $d+3$, uses $m$ labels, and has size more than $N$, by Theorem~\ref{thm:reduce} there is a node $v$ of $T$ that has nodes $v_1,\ldots,v_M$ among its children such that for the tree $T':= T - (T_{v_1} \cup \ldots \cup T_{v_M})$ we have $T' \models csd_d^{tree}$. 
This means that the matrix $A(T')$ has contraction$^*$-depth $d$ and consequently $\csd(M(A(T'))) = d$. To obtain a contradiction with the assumptions on $M$, it remains to argue that $M(A(T'))$ is a restriction of $M$.
Let $T''$ be the tree obtained from $T$ by removing the leaves of trees $T_{v_1},\ldots , T_{v_M}$ marked with label $L_C$. By repeated application of Lemma~\ref{lem:ignore_rows} we have that $M(A(T''))$ is isomorphic to $M(A(T'))$, and so $\csd(M(A(T''))) = d$. Since $A(T'')$ is obtained from $A = A(T)$ by removing columns, $M(A(T''))$ is a restriction of $M$ with $\csd(M(A(T''))) = d$, a contradiction.
\end{proof}


\begin{proof}[Proof of Theorem~\ref{thm:main1}]
    Let $M$ be a $\F$-represented matroid with $\csd (M) = d$ and $\csd (M \setminus e) < d$ for every $e \in M$. Let $S$ be any non-empty set of elements of $M$. Consider any $e \in S$. Then $M \setminus S$ is a restriction of $M \setminus e$. By our assumption on $M$ we have that $\csd(M\setminus e) < d$. Since contraction$^*$-depth is closed under deletion, we have $\csd(M\setminus S) \le \csd(M\setminus e)$ and so $\csd(M\setminus S) < d$. Thus, by Lemma~\ref{lem:main1} we have $|M| < f(|\FF|,d)$.
\end{proof}

\section{Deletion$^*$-depth}

In this section we define a new depth parameter for representable matroids called the \emph{deletion$^{*}$-depth}, and investigate its properties.

\begin{definition}
\label{def:dsd}  
Let $M = M(A)$ be a represented matroid and let $h$ be the number of rows of $A$. The \emph{deletion$^*$-depth} $\dsd(M)$ of $M$ is defined recursively as:
\begin{itemize}
    \item If $r(M) = |M|$, then $\dsd(M) = 0$. 
    \item If $M$ is not connected, then the deletion$^*$-depth of $M$ is the maximum deletion$^*$-depth of a component of $M$.
    \item If $M$ is connected, then $\dsd(M) = 1 + \min \dsd(M(A \oplus \ve{v^\top}))$, where the minimum is taken over all $\ve{v} \in \F^h$ and $\oplus$ is the operation that adds a row to $A$.
\end{itemize}
\end{definition}


We claim that deletion$^*$-depth is a natural counterpart of contraction$^*$-depth. We base this on the following:
\begin{itemize}
    \item For every represented matroid $M$ we have $\csd(M) = \dsd(M^*)$. This is the core technical result of this section.
    \item Deletion$^*$-depth is functionally equivalent to deletion-depth. More precisely, we will show that for any represented matroid we have $\dsd(M) \le \dd(M) \le f(\dsd(M))$ for some function $f$. This follows easily from the analogous result for $\cd$ and $\csd$ and duality (see Theorem~\ref{thm:dd_dsd_fequivalent}).
\end{itemize}

Before we proceed with proving that $\csd(M) = \dsd(M^*)$, let us briefly comment on the relationship between the operation of deletion (that in the case of represented matroids deletes a column of $A$) and operation $\oplus$ used in the definition of deletion$^*$-depth. These two operations are quite different, and it may seem counterintuitive that $\oplus$ should generalize the deletion operation (so that we get $\dsd(M) \le \dd(M)$). However, we can simulate deletion of a column by $\oplus$ as follows. Let $\ve{u}$ be the column of $A$ that we wish to remove; say $\ve{u}$ is the $i$-th column of $A$. Consider the vector $\ve{e_i} \in \F^n$ that has $0$ everywhere except for the $i$-th position, where it has $1$. It is easily seen that $M(A \oplus \ve{e_i})\setminus \ve{u'}$ is isomorphic to $M(A)\setminus \ve{u}$, where $\ve{u'}$ is the $i$-th column of $A \oplus \ve{e_i}$. Moreover, note that even though the column $\ve{u'}$ is not removed from the matrix $A \oplus \ve{e_i}$, it is now disconnected from the rest of the matroid, and so from this point on it will be treated separately by the definition of deletion$^*$-depth.


\subsection{Interlude -- matroids via linear subspaces}

In what follows, it will be convenient to think of $\F$-representable matroids with ground set of size $n$ as of linear subspaces of $\F^n$. To be more explicit, if $A$ is a matrix representing a matroid $M$, we will associate with $M$ the linear subspace $W = \Span(\rows(A))$ of $\F^n$. This correspondence between $\F$-representable matroids with $n$ elements and subspaces of $\F^n$ works also in the other direction -- to each subspace $W$ of $\F^n$ we can meaningfully associate matroid $M(W)$, as we show below. We will need the following lemma.

\begin{lemma}
\label{lem:basis_independent}
    Let $W$ be a subspace of $\F^n$. Let $\ve {v_1},\ldots, \ve {v_k}$ and $\ve {u_1},\ldots, \ve {u_k}$ be two bases of $W$. Let $A$ be the matrix with rows $\vev_1^{\top},\ldots, \ve {v_k}^{\top}$  and $A'$ be the matrix with rows $\ve {u_1}^{\top},\ldots, \ve {u_k}^{\top}$. Then $M(A)$ is isomorphic to $M(A')$.
\end{lemma}
\begin{proof}
The transformation of $A$ into $A'$ can be obtained by a sequence of elementary row operations. It is easy to check that each elementary row operation preserves the matroid structure.
\end{proof}

Next, for any linear subspace $W$ of $\F^n$ we define matroid $M(W)$. The definition of $M(W)$ relies on using a basis of $W$, but by Lemma~\ref{lem:basis_independent} the matroid $M(W)$ itself is independent of the choice of basis.
\begin{definition}
    Let $W$ be a linear subspace of $\F^n$. We define the matroid $M(W)$ as follows. Let $\ve {v_1},\ldots, \ve {v_k}$ be a basis of $W$. Let $A$ be a matrix with rows $\ve {v_1}^{\top},\ldots, \ve{v_k}^{\top}$. Then $M(W) := M(A)$.
\end{definition}

Next we show that the dual of $M(W)$ has a very natural representation in terms of subspaces -- it is represented by $W^\bot$, the subspace of $\F^n$ orthogonal to $W$. 

\begin{lemma}
\label{lem:dual_subspaces}
    For any subspace $W$ of $\F^n$, we have $M(W)^* \cong M(W^\bot)$
\end{lemma}
\begin{proof}
    Let $\ve {w_1},\ldots, 
    \ve {w_k}$ be a basis of $W$ and let $\ve {w_{k+1}},\ldots, \ve{w_n}$ be a basis of $W^\bot$.   
Let $A$ be the matrix with rows $\ve {w_1}^\top,\ldots, \ve{w_k}^\top$ and let $B$ be a matrix with rows $\ve {w_{k+1}}^\top,\ldots, \ve{w_n}^\top$.
 Then we have $M(W) = M(A)$ and our goal is to prove that $M(W)^* = M(B)$. Note that since the row rank of $A$ is $k$, the column rank of $A$ is also $k$, and hence $r(M(W)) = r(M(A))= k$. Also note that $r(B) = n-k$.

In what follows we will index columns of $A$ and $B$ by numbers from $[n]$. This sets up a natural correspondence between the columns of $A$ and $B$: the $i$-th column of $A$ corresponds to the $i$-th column of $B$. We will use the following notation: if  $U \subset [n]$ and $C$ is a matrix that has as its rows transposed vectors from $\F^n$, then by $C[U]$ we mean the submatrix of $C$ given by the columns with indexes from $U$.
In particular, if $U = \{i\}$ is a singleton and $\ve{u} \in F^n$, then $\ve{u}^\top[\{i\}]$ is the $i$-th entry of $\ve{u}$. In this case we will write $\ve{u}^\top[i]$ instead of  $\ve{u}^\top[\{i\}]$

Let $S \subseteq [n]$ be a set such that the columns of $A[S]$ form a basis of $M(A)$ (equivalently, a basis of $\F^k$).
 Let $\bar{S}$ be the complement of $S$ in $[n]$. Our goal is to show that $B[\bar{S}]$ has rank $n-k$; this will mean that the columns of $B[\bar{S}]$ form a basis of n $M(W^\bot)$, which will confirm that $M(W)^* = M(B)$. 
Assume for contradiction that $B[\bar{S}]$ has rank strictly less than $n-k$. Since $B[\bar{S}]$ has $n-k$ rows, this means that the set of rows of $B[\bar{S}]$ is linearly dependent. 
Therefore, there exists a vector $\ve{u} \in \F^n$ such that $\ve{u}^\top \in \Span(\rows(B))$ and $\ve{u}^\top[\bar{S}] = \ve{0}$. Moreover, we know that $\ve{u}\not=\ve{0}$, since $\ve{u}^\top \in \Span(\rows(B))$ and $B$ has rank $n-k$ and $n-k$ rows. Therefore, there exists  $i \in S$ such that $\ve{u}^\top[i] = \alpha$ with $\alpha \not= 0$. Let $\ve{v}^\top$ be a vector in $\Span(\rows(A))$ such that $\ve{v}^\top[i] = 1$, $\ve{v}^\top[j] = 0$ for all $j \in S \setminus \{i\}$, and with arbitrary value $\ve{v}^\top[j']$ for any $j' \in \bar{S}$; such  $i \in S$ exists since $A[S]$ is of full rank. Then we have that $\ve{u} \cdot \ve{v} = \alpha$. This is a contradiction, since from $\ve{v} \in W$ and $\ve{u} \in W^{\bot}$ we have that $\ve{u} \cdot \ve{v} = 0$.
\end{proof}

For a subspace $W$ of $\F^n$, there are two very natural operations -- `adding' a vector $\ve{v} \in \F^n$ to $W$ (this can be realized as replacing $W$ with $\Span(W \cup \ve{v})$) and `removing' a vector from $W$ (this can be realized as replacing $W$ by $W \cap \Span(\ve{v})^\bot$). It can be easily checked that these two operations are dual to each other -- see also below.
Later we will see that if we consider the matroid $M(W)$, these operations correspond to deletion$^*$ and contraction$^*$, respectively.

\begin{definition}[Contraction$^\bullet$ on a subspace]
Let $W$ be a subspace of $\F^n$ and $\ve{v} \in \F^n$. We define $M(W)/^\bullet \ve{v} = M(W \cap \Span(\ve{v})^\bot)$.
\end{definition}

\begin{definition}[Deletion$^\bullet$ on a subspace]
Let $W$ be a subspace of $\F^n$ and $\ve{v} \in \F^n$. We define $M(W)\setminus^{\hspace{-0.2em}\bullet} \ve{v} = M(\Span(W \cup \{\ve{v}\}))$.
\end{definition}

A simple formal argument shows that these two operations are dual to each other.
\begin{theorem}
\label{thm:cb_db_dual}
    Contraction$^\bullet$ is dual to deletion$^\bullet$, i.e. for any matroid $M(W)$ and any vector $\ve{v}$ we have
    $(M(W)/^\bullet \ve{v})^* = M(W)^*\setminus^{\hspace{-0.2em}\bullet} \ve{v}$.
\end{theorem}
\begin{proof}
We have 
    $$(W \cap \Span(\ve{v})^\bot)^\perp = \Span(W^\bot \cup \Span(\ve{v})^{\bot \bot}) = \Span(W^\bot \cup \Span(\ve{v})) = \Span(W^\bot \cup \{\ve{v}\}).$$

    
Using this and Lemma~\ref{lem:dual_subspaces} we have that

$$   (M(W)/^\bullet \ve{v})^* = M(W \cap \Span(\ve{v})^\bot)^* = M((W \cap \Span(\ve{v})^\bot)^\bot) = M(\Span(W^\bot \cup \{\ve{v}\})) = $$
    $$ \pushQED{\qed} 
    = M(W^\bot)\setminus^{\hspace{-0.2em}\bullet} \ve{v} = M(W)^*\setminus^{\hspace{-0.2em}\bullet} \ve{v}   
    \qedhere
    \popQED
    $$    
\renewcommand{\qedsymbol}{}
\end{proof}

Let $M=M(W)$ be a matroid given by a subspace $W$ of $\F^n$.
We can now define two auxiliary depth parameters $\cbd(M)$ and $\dbd(M)$ as follows. 
We set $\cbd(M) := 0$ if $r(M) =0$. If $M$ is disconnected, then $\cbd(M)$ is the maximum $\cbd$ of any connected component of $M$. If $M$ is connected and $r(M) > 0$, then $\cbd(M)$ is the minimum of $\cbd(M/^\bullet \ve{v})$ over all $\ve{v} \in \F^n$.
For $\dbd$ we set $\dbd(M) := 0$ if $r(M) =n$. If $M$ is disconnected, then $\dbd(M)$ is the maximum $\dbd$ of any connected component of $M$. If $M$ is connected and $r(M) > 0$, then $\dbd(M)$ is the minimum of $\dbd(M \setminus ^{\hspace{-0.2em}\bullet} \ve{v})$ over all $\ve{v} \in \F^n$.

From Theorem~\ref{thm:cb_db_dual} and Lemma~\ref{lem:dual_subspaces}  we easily get the following.
\begin{lemma}
\label{lem:cbd_dbd_dual}
    For any matroid $M= M(W)$, we have $\cbd(M) = \dbd(M^*)$ and $\cbd(M^*) = \dbd(M)$.
\end{lemma}

For a represented matroid $M = M(A)$ we set $\cbd(M) := \cbd(\Span(\rows(A)))$ and $\dbd(M) := \dbd(\Span(\rows(A)))$. Then of course the results from Lemma~\ref{lem:cbd_dbd_dual} apply to represented matroids.

Before we proceed further, let us briefly comment on how the operations $/^\bullet$ and $\setminus^{\hspace{-0.2em}\bullet}$ are realized in the case of represented matroids. If $M = M(A)$ is a represented matroid, we can obtain a representation for $M\setminus^{\hspace{-0.2em}\bullet} \ve{v}$ by simply adding the row $\ve{v}^\top$ to $A$ to obtain matrix $A'$. It is easily seen that $M\setminus^{\hspace{-0.2em}\bullet} \ve{v} = M(A')$, because if $W = \Span(\rows(A))$, then $\Span(W \cup \{\ve{v}\}) = \Span(\rows(A) \cup \{\ve{v}\}) = \Span(\rows(A'))$. For the operation $/^\bullet$, we can obtain a representation of $M/^\bullet \ve{v}$ as follows. We consider two cases. First, if $\ve{v} \in \Span(\rows(A))$, then we pick an orthogonal basis $\ve{v_1},\ldots,\ve{v_k}$ of $\Span(\rows(A))$ such that $\ve{v_1} = \ve{v}$ and consider matrix $A'$ consisting of rows $\ve{v_2},\ldots,\ve{v_k}$. Then one easily checks that for every $\ve{w}$ we have $\ve{w} \in \Span(\rows(A)) \cap \ve{v}^\bot$ if and only if $\ve{w} \in \Span(\ve{v_2},\ldots,\ve{v_k})$, and so $A'$ is a representation for $M\setminus^{\hspace{-0.2em}\bullet} \ve{v}$. Second, if $\ve{v} \not\in \Span(\rows(A))$, we can write $\ve{v}$ uniquely as $\ve{v'} + \ve{v''}$, where $\ve{v'} \in \Span(\rows(A))$ and $\ve{v''} \in \Span(\rows(A))^\bot$.
The we have $\Span(\rows(A)) \cap \ve{v}^\bot = \Span(\rows(A)) \cap \ve{v'}^\bot$ (see Section~\ref{sec:linalg}) and we can proceed as before with $\ve{v_1} = \ve{v'}$.

\subsection{Duality of deletion$^*$-depth and contraction$^*$-depth}

In this section we prove Theorem~\ref{thm:csd_dsd_dual}, stating that for every represented matroid $M=M(A)$ we have $\csd(M) = \dsd(M^*)$. To this end, we will show that for every represented matroid we have $\csd(M) = \cbd(M)$ and $\dsd(M) = \dbd(M)$. Since we know from the previous section that $\cbd(M) = \dbd(M^*)$, we will then get $\csd(M) = \cbd(M) = \dbd(M^*) = \dsd(M^*)$, as desired.

Next, we show that the contraction$^*$ operation is the same as the newly defined operation of contraction$^\bullet$ on a subspace. In the following lemma we will assume that a represented matroid $M$ we are working with is represented by a matrix of full rank. One easily checks that that every representable matroid can be represented in this way.

\begin{lemma}
\label{lem:cs_cb_equal}
    Let $A \in \F^{k\times n}$ be a matrix of rank $k$ and $W = \Span(\rows(A))$. Let $M = M(A)$. For every $\ve{u} \in \F^k$ there exists $\ve{v} \in  \F^n$ such that $M/\ve{u} = M/^\bullet \ve{v}$ and conversely, for every $\ve{v} \in \F^n$ there exists $\ve{u} \in \F^k$ such that $M/\ve{u} = M/^\bullet \ve{v}$.
\end{lemma}
\begin{proof}
Fist note that it is enough to consider the situations when  $\ve{v} \in W$.  This is because if $\ve{v} \not\in W$, then we can write $\ve{v}$ as $\ve{v'} + \ve{v''}$, where $\ve{v'} \in W$ and $\ve{v''} \in W^\bot$. But then we have $M/^\bullet \ve{v} = M/^\bullet\ve{v'}$ (see the discussion at the end of the previous section) and so it is enough to consider $\ve{v'} \in W$ instead of $\ve{v}$.

We will now proceed with the proof of the lemma assuming that  $\ve{v} \in W$. The proof will be based on two observations.

Observation 1: 
 Let $T$ be a $k \times k$ matrix of rank $k$. Further, let
  $\ve{u_1}$ be the first column of $T$ and set $A' := T^{-1}A$.
Let $S = \{\ve{w_1},\ldots,\ve{w_m}\}$ be a set of columns of $A$ and let $S' = \{\ve{w'_1},\ldots,\ve{w'_m}\}$ be the corresponding set of columns of $A'$ (so we have $\ve{w_i'} = T^{-1}\ve{w_i}$). 
We claim that for every $\alpha_1,\ldots,\alpha_m \in \F$ we have
\begin{equation}
\label{eq:aa}
 \alpha_1 \ve{w_1} + \ldots + \alpha_m \ve{w_m} \in \Span(\ve{u_1}) \Longleftrightarrow \alpha_1 \ve{w_1'} + \ldots + \alpha_m \ve{w_m'} \in \Span(\ve{e_1})    
\end{equation}
where $\ve{e_1} = (1,0,\ldots, 0) \in \F^k$.

To see this, assume that $\alpha_1 \ve{w_1} + \ldots + \alpha_m \ve{w_m} = \beta \ve{u_1}$ for some $\beta \in \FF$. Then, by applying $T^{-1}$ to both sides of this equality we get $\alpha_1 \ve{w_1'} + \ldots + \alpha_m \ve{w_m'} = \beta \ve{e_1}$. For the other direction, we assume that $\alpha_1 \ve{w_1'} + \ldots + \alpha_m \ve{w_m'} = \beta \ve{e_1}$ for some $\beta \in \FF$ and apply $T$ to both sides. Then, analogously to the previous case, we get $\alpha_1 \ve{w_1} + \ldots + \alpha_m \ve{w_m} = \beta \ve{u_1}$.

Observation 2: Let $A''$ be the matrix obtained from $A'$ by removing the first row. Let $S' = \{\ve{w_1'},\ldots, \ve{w_m'}\}$ be a set of columns of $A'$ and let $S'' = \{\ve{w_1''},\ldots,\ve{w_m}''\}$ be the corresponding set of columns of $A''$.  Then for every $\alpha_1,\ldots,\alpha_k \in \F$ we have that 
\begin{equation}
\label{eq:aaa}
 \alpha_1 \ve{w_1'} + \ldots + \alpha_m \ve{w_m'} \in \Span(\ve{e_1}) \Longleftrightarrow
 \alpha_1 \ve{w_1''} + \ldots + \alpha_m \ve{w_m''} = \ve{0}     
\end{equation}

We now prove that for every $\ve{u} \in \F^k$ there exists a corresponding $\ve{v} \in W$ as in the statement of the lemma. Let $\ve{u_1},\ldots,\ve{u_k}$ be a basis of $\F^k$ such that $\ve{u_1} = \ve{u}$. Let $T$ be the matrix with columns $\ve{u_1},\ldots,\ve{u_k}$.  Set $A' := T^{-1} A$ and let $A''$ be obtained from $A'$ by removing its first row $\ve{v_1}$.
Let $S =\{\ve{w_1},\ldots, \ve{w_m}\}$ be any set of columns of $A$ and $\alpha_1,\ldots,\alpha_m \in \F$ arbitrary. Then by (\ref{eq:aa}) and (\ref{eq:aaa}) we have that that 
\begin{equation}
\label{eq:aaaa}
 \alpha_1 \ve{w_1} + \ldots + \alpha_m \ve{w_m} \in \Span(\ve{u}) \Longleftrightarrow \alpha_1 \ve{w_1''} + \ldots + \alpha_m \ve{w_m''} = \ve{0}
\end{equation}
and so $A''$ is a representation of $M/\ve{u}$. We now argue that $A''$ is also a representation of $M/^\bullet \ve{v}$ for some $\ve{v} \in W$. By construction it is easy to see that $\Span(\rows(A')) = W$ and 
that  $A'$ has $k$ rows, so after removing its first row $\ve{v_1}$ we know that $W '' = \Span(\rows(A''))$ is a $(k-1)$-dimensional subspace of $W$. Hence we can pick $\ve{v} \in W$ such that $W'' = W \cap \Span(\ve{v})^\bot$. Consequently, we have that $M(A'') = M(W \cap \Span(\ve{v})^\bot) = M/^\bullet\ve{v}$. Thus, we have $M/\ve{u} = M(A'') = M/^\bullet\ve{v}$, as desired.

For the other direction, we prove that for every $\ve{v} \in W$ there exists $\ve{u} \in \F^k$ as in the statement of the lemma. Let $\ve{w_1'},\ldots, \ve{w_k'}$ be an orthogonal basis of $W$ with $\ve{w_1'} = \ve{v}$. Let $A'$ be the matrix with rows $\ve{w_1'}^\top,\ldots, \ve{w_k'}^\top$ and let $A''$ be obtained from $A'$ by removing its first row $\ve{v^\top}$. Then we have that the rows of $A''$ are a basis of $W \cap \Span(\ve{v})^\bot$, and so $M(A'') = M/^{\bullet} \ve{v}$. 
Since $rk(A) = rk(A')$, there exists an invertible matrix $T$ such that $A = TA'$, and so we have $A' = T^{-1}A$.
 Let $\ve{u}$ be the first column of $T$. Let $S'' = \{\ve{w_1''},\ldots, \ve{w_m''}\}$ be a set of vectors of $A''$ and let $S = \{\ve{w_1},\ldots, \ve{w_m}\}$ be the corresponding set of columns of $A$. Let $\alpha_1,\ldots, \alpha_k \in \F$ be arbitrary. By (\ref{eq:aaa}) and (\ref{eq:aa}) we again get~(\ref{eq:aaaa}), and so $A''$ is a representation of $M/\ve{u}$. Thus, we again have $M/\ve{u} = M(A'') = M/^\bullet\ve{v}$, as desired.
    
\end{proof}

As an immediate consequence of Lemma~\ref{lem:cs_cb_equal} we obtain that for every represented matroid $M$ we have $\csd(M)= \cbd(M)$. This finishes the proof of Theorem~\ref{thm:csd_dsd_dual}.


\subsection{Functional equivalence of deletion-depth and deletion$^*$-depth}

\begin{theorem}
\label{thm:dd_dsd_fequivalent}
    There exists a function $f$ such that $\dsd(M) \leq \dd(M) \leq f(\dsd(M))$ for any represented matroid $M$.
\end{theorem}
\begin{proof}
We have 
    $$ \dsd(M) = \csd(M^*) \le \cd(M^*) = \dd(M)$$
by Theorem~\ref{thm:csd_dsd_dual}, Theorem~\ref{thm:cd_csd_fequivalent} and Lemma~\ref{lem:cd_dd_dual}, respectively. By the same results (in reverse order) we have 
$$ \dd(M) = \cd(M^*) \le f(\csd(M^*)) = f(\dsd(M^*))$$
where $f$ is the function from Theorem~\ref{thm:cd_csd_fequivalent}.
\end{proof}

\subsection{Minimal matroids of bounded deletion$^*$-depth}

We now use Theorem~\ref{thm:csd_dsd_dual} to prove Theorem~\ref{thm:main_dsd}.
\begin{proof}[Proof of Theorem~\ref{thm:main_dsd}]
Let $M$ be a matroid as in the assumptions of the theorem. For contradiction assume that $|M| > f(|\F|, d)$, where $f$ is the function from Theorem~\ref{thm:main1}. By Theorem~\ref{thm:csd_dsd_dual} we have that $\csd(M^*) = d$. Since $|M^*| > f(|\F|, d)$, there exists $e \in M$ such that  $\csd(M^*\setminus e) = d$. By using Theorem~\ref{thm:csd_dsd_dual} once more we then have $\dsd((M^*\setminus e)^*) = d$. Finally, by the duality of deletion and contraction we have $(M^*\setminus e)^* = M^{**}/e = M/e$, and so we get that $\dsd(M/e) =d$, which is a contradiction with our assumption on $M$.
\end{proof}

\section{Minimal matroids for contraction-depth and deletion-depth}

The proof of Theorem~\ref{thm:main1} can be easily adapted to give a proof of Theorem~\ref{thm:main_cd}. Recall that this theorem states that if $M$ is an $\mathbb{F}$-represented matroid such that
$\cd (M) = d$ and $\cd (M \setminus S) < d$ for every non-empty subset $S$ of elements of $M$, then $|M| \leq f(|\FF|, d)$ for some function $f$.
We can follow the same proof as in Theorem~\ref{thm:main1} all the way up to Lemma~\ref{lem:main1}, adjusting what is necessary.
We only sketch the steps below, as they are analogous than the proof of Theorem~\ref{thm:main1}. We proceed as follows:
\begin{enumerate}
    \item We start with a $\F$-represented matroid $M = M(A)$ with $\cd(M) = d$, and for contradiction assume that $|M| > N$ for a suitably chosen $N$ depending on $|\F|$ and $d$.
    \item We show that for every $d$ there exists a CMSO formula $cd_d$ such that for every matrix $A$ we have that $\S(A) \models cd_d$ if and only if $cd(A) = d$. This formula can be constructed the same way as in the proof of Theorem~\ref{thm:csd_CMSO}, with the main difference being that we do not need to use virtual columns (alternatively, we can require that each virtual column corresponds to an actual column of $A$). Apart from this, we only need to adjust the basis of the definition to reflect that a matroid with one element has contraction-depth $1$.
    \item By Theorem~\ref{thm:cd_csd_fequivalent} we have  $\csd(M) \le g(\cd(M))$ for some function $g$, and by Corollary~\ref{cor:dual_td} we know that $M$ can be represented by matrix $A'$ of dual treedepth at most $g(\cd(M))$. 
    \item From this point on, we continue the same way as in Sections~\ref{sec:encoding_matrices} and~\ref{sec:proof_finish}. The result of this is a variant of Lemma~\ref{lem:main1} with $\csd$ replaced by $\cd$, which is precisely Theorem~\ref{thm:main_cd}.
\end{enumerate}

From Theorem~\ref{thm:main_cd} we easily obtain Theorem~\ref{thm:main_dd} by duality. Recall that this theorem states that if $M$ is an $\mathbb{F}$-represented matroid such that
$\dd (M) = d$ and $\dd (M / e) < d$ for every $e \in M$, then $|M| \leq f(|\FF|, d)$ for some function $f$.

\begin{proof}[Proof of Theorem~\ref{thm:main_dd}]
Let $M$ be a matroid as in the assumptions of the theorem. 
For contradiction assume that $|M| > f(|\F|, d)$, where $f$ is the function from Theorem~\ref{thm:main_cd}. By Lemma~\ref{lem:cd_dd_dual} we have that $\cd(M^*) = d$. Since $|M^*| > f(|\F|, d)$, there exists a non-empty set $S$ of elements of $M^*$  $\cd(M^*\setminus S) = d$. By using Lemma~\ref{lem:cd_dd_dual} again, we have $\dd((M^*\setminus S)^*) = d$. Finally, by the duality of deletion and contraction we have $(M^*\setminus S)^* = M^{**}/S = M/S$, and so we obtain that $\dd(M/S) = d$, which is a contradiction with our assumption on $M$.
\end{proof}

\section{Further results}

 Let $M$ be a matroid represented by a matrix $A$ with $m$ rows. Recall that we call a vector $\ve{v} \in \F^m$ \emph{progresive} if $\csd(M/\ve{v}) < \csd(M)$. Theorem~\ref{thm:few_moves} states that there exists a function $g: \N \times \N \to \N$ with the following property: For every finite field $\F$ and $d \in \N$, if $M$ is a $\F$-represented matroid with $\csd(M) = d$, then the number of progressive vectors for $M$ is at most $g(|\FF|, d)$.
\begin{proof}[Proof of Theorem~\ref{thm:few_moves}]
    Let $A$ be the matrix representing $M$. By Theorem~\ref{thm:main1} $M$ contains a submatroid $M'$ with $\csd(M') = d$ and with at most $f(|\FF|,d)$ elements. Let $S$ be the set of columns of $A$ that correspond to the elements of $M'$. 

    We claim that every progressive vector $\ve{v}$ for $M$ is in $\Span(S)$. Assume for contradiction that this is not the case. Then the matroid $M'/\ve{v}$ has contraction$^*$-depth $d$, because any subset of $S$ is dependent in $M'/\ve{v}$ if and only if it is dependent in $M'$. But $\csd(M/\ve{v}) < d$ since $v$ is progressive, and by Lemma~\ref{lem:restriction} we have $\csd((M'/\ve{v}) \le   \csd(M/\ve{v})$, so $\csd((M'/\ve{v}) < d$, a contradiction.
    

    Consequently, since every progressive vector is a linear combination of vectors from $S$, it is  determined by $|S|$ coefficients from $\F$. Therefore, there are at most $\F^{|S|}$ such vectors. Thus, we can set $g(|\FF|, d):= |\F|^{f(|\FF|, d)}$, where $f$ is the function from Theorem~\ref{thm:main1}.
\end{proof}


\subsection{Results for $\Q$-represented matroids}
Recall that by Lemma~\ref{lem:rational_to_finite}  every $\Q$-represented matroid $M$ with $\csd(M) =d$ and  entry complexity $k$ can be $\F_q$-represented with $q \le f(d,k)$ for some function $f$. Then we get the following two results as an easy corollary of Theorem~\ref{thm:main1} and Theorem~\ref{thm:few_moves}. 
\begin{theorem}
\label{thm:main_rational}
Let $M$ be an $\mathbb{Q}$-represented matroid with entry complexity $k$ such that
$\csd (M) = d$ and $\csd (M \setminus e) < d$ for every $e \in M$. Then $|M| \leq f(d, k)$ for some function $f$.
\end{theorem}

\begin{theorem}
\label{thm:few_moves_rational}
    Let $M$ be a $\Q$-represented matroid with $\csd(M) = d$ with entry complexity $k$. Then the number of progressive vectors for $M$ is at most $g(d, k)$ for some function $g$.
\end{theorem}





\bibliography{biblio}
\bibliographystyle{abbrv}

\end{document}